\documentclass{elsarticle}

\usepackage{etoolbox} 

\makeatletter
\def\ps@pprintTitle{%
\let\@oddhead\@empty
\let\@evenhead\@empty
\def\@oddfoot{}%
\let\@evenfoot\@oddfoot}
\makeatother
\usepackage{xcolor}

\usepackage{amsmath}
\usepackage{cases}
\usepackage{color} 
\usepackage{amssymb}
\usepackage{amsthm}
\usepackage{MnSymbol}
\usepackage{pgf,tikz}
\usepackage{caption}

\usepackage{chngcntr}
\usepackage{lipsum}
\usepackage{color}
\usetikzlibrary{arrows}

\theoremstyle{plain}
\newtheorem{thm}{Theorem}[section]
\newtheorem{lem}{Lemma}[section]
\newtheorem{cor}{Corollary}[section]

\newtheorem{prop}{Proposition}[section]

\newtheorem{Fact}{Fact}[section]

\newtheorem{claim}{Claim}[section]
\newtheorem{Notation}{Notation}[section]

\newtheorem{Observation}{Observation}[section]
\numberwithin{equation}{section}

\begin{document}
\nocite{*}
\sloppy

\newtheorem*{Convention}{Convention}

\begin{frontmatter}

\title{Making a tournament indecomposable by one subtournament-reversal operation}
\author{Houmem Belkhechine\fnref{}}
\address{University of Carthage, Bizerte Preparatory Engineering Institute, Bizerte, Tunisia}
\ead{houmem.belkhechine@ipeib.rnu.tn}

\author{Cherifa Ben Salha\fnref{}}
\address{University of Carthage, Faculty of Sciences of Bizerte, Bizerte, Tunisia}
\ead{cherifa.bensalha@fsb.u-carthage.tn}

\begin{abstract}     
Given a tournament $T$, a module of $T$ is a subset $M$ of $V(T)$ such that for $x, y\in M$ and $v\in V(T)\setminus M$, $(v,x)\in A(T)$ if and only if $(v,y)\in A(T)$. The trivial modules of $T$ are $\emptyset$, $\{u\}$ $(u\in V(T))$ and $V(T)$. The tournament $T$ is indecomposable if all its modules are trivial; otherwise it is decomposable. Let $T$ be a tournament with at least five vertices. In a previous paper, the authors proved that the smallest number $\delta(T)$ of arcs that must be reversed to make $T$ indecomposable satisfies $\delta(T) \leq \left\lceil \frac{v(T)+1}{4} \right\rceil$, and this bound is sharp, where $v(T) = |V(T)|$ is the order of $T$.   
In this paper, we prove that if the tournament $T$ is not transitive of even order, then $T$ can be made indecomposable by reversing the arcs of a subtournament of $T$. We denote by $\delta'(T)$ the smallest size of such a subtournament. We also prove that $\delta(T) = \left\lceil \frac{\delta'(T)}{2} \right\rceil$.               
\end{abstract}
\end{frontmatter}

\section{Introduction and main results}
A {\it tournament} $T = (V(T), A(T))$ consists of a finite set $V(T)$ of {\it vertices} together with a set $A(T)$ of ordered pairs of distinct vertices, called {\it arcs}, such that for every $x \neq y \in V(T)$, $(x,y) \in A(T)$ if and only if $(y,x) \not\in A(T)$. The {\it order} of $T$, denoted by $v(T)$, is the number of its vertices. Given a tournament $T$, the {\it subtournament} of $T$ induced by a subset $X$ of $V(T)$ is the tournament $T[X] = (X, A(T) \cap (X \times X))$. For $X\subseteq V(T)$, the subtournament $T[V(T) \setminus X]$ is also denoted by $T-X$, and by $T-x$ when $X= \{x\}$. Two tournaments $T$ and $T'$ are {\it isomorphic} if there exists an {\it isomorphism} from $T$ onto $T'$, i.e., a bijection $f$ from $V(T)$ onto $V(T')$ such that for every $x, y \in V(T)$, $(x, y) \in A(T)$ if and only if $(f(x), f(y)) \in A(T')$. 
A {\it transitive} tournament is a tournament $T$ such that for every $x,y,z \in V(T)$, if $(x,y) \in A(T)$ and  $(y,z) \in A(T)$, then $(x,z) \in A(T)$. Given a positive integer $n$, every transitive tournament of order $n$ is isomorphic to the transitive tournament $\underline{n} = (\{0, \ldots, n-1\}, \{(i,j) : 0 \leq i < j \leq n-1\})$.

The classic notion of a module is the main object of the paper. Given a tournament $T$, a subset $M$ of $V(T)$ is a {\it module} \cite{Spinrad} (or a {\it clan} \cite{E} or an {\it interval} \cite{I}) of $T$ provided that for every $x,y \in M$ and for every $v \in V(T) \setminus M$, $(v,x) \in A(T)$ if and only if $(v,y) \in A(T)$. For example, $\emptyset$, $\{x\}$, where $x \in V(T)$, and $V(T)$ are modules of $T$, called {\it trivial} modules. 
A tournament is {\it indecomposable} \cite{I, ST} (or {\it prime} \cite{Spinrad} or {\it primitive} \cite{E} or {\it simple} \cite{Erdos}) if all its modules are trivial; otherwise it is {\it decomposable}. An isomorphism preserves modules, in particular it preserves indecomposability. Let us consider the tournaments of small orders. The tournaments of orders at most $2$ are clearly indecomposable. Up to isomorphism, the tournaments of order $3$ are the decomposable tournament $\underline{3}$, and the indecomposable tournament $(\{0,1,2\}, \{(0,1), (1,2), (2,0)\})$. Up to isomorphism, there are four tournaments of order $4$, each of them is decomposable. Similarly, the transitive tournaments of orders at least 3 are all decomposable. More precisely, for every integer $n \geq 3$,
 the modules of  $\underline{n}$ are the intervals of the usual total order on $V(\underline{n})$. On the other hand, it is well-known that for every integer $n \geq 5$, there exist indecomposable tournaments of order $n$ (see e.g., Theorem~\ref{thmerdos} below).

 Our topic is based on the following question $(Q)$: what is the minimum number of some allowed operations that must be successively applied to a tournament in order to make it indecomposable ? J.W. Moon \cite{Moon} and P. Erd\H{o}s et al. \cite{Erdos} studied this question when an allowed operation, that we call {\it vertex-addition} operation, consists of adding a single vertex. They proved that given a tournament $T$ of order at least $4$, only one vertex-addition operation suffices to make $T$ indecomposable, unless $T$ is transitive of odd order. Formally: 

\begin{thm} [\cite{Erdos, Moon}] \label{thmerdos}
 Given a tournament $T$ of order at least $4$ that is not transitive of odd order, there exists an indecomposable tournament $T'$ such that $T'-v = T$ for some $v \in V(T')$. 
\end{thm}

Theorem~\ref{thmerdos} is at the origin of similar studies in graphs, digraphs,  and other combinatorial structures: the question is to determine the minimum number of vertices that must be added to a given structure in order to make it indecomposable (see e.g., \cite{BI, BIW, Brignall, BRV}).  

In this paper, we are interested in Question $(Q)$ when an allowed operation consists of reversing arcs according to given rules. We will see how Theorem~\ref{thmerdos} can be stated in this context (see Theorem \ref{thmerdos2} below). We need some notations and terminology. 
Let $T$ be a tournament. An {\it arc-reversal} operation consists of reversing a single arc $a = (x,y) \in A(T)$, i.e., replacing the arc $a$ by $a^{\star} = (y,x)$ in $A(T)$. The tournament obtained from $T$ after reversing the arc $a$ is denoted by ${\rm Inv}(T,a)$ or ${\rm Inv}(T,\{x,y\})$. Thus ${\rm Inv}(T,a) = {\rm Inv}(T,\{x,y\}) = (V(T), (A(T) \setminus \{a\}) \cup \{a^{\star}\})$. The reversal of a subset $B$ of $A(T)$ from the tournament $T$ is done as a succession of arc-reversal operations. The tournament obtained from $T$ after reversing $B$ is denoted by ${\rm Inv}(T,B)$. Thus ${\rm Inv}(T,B) = (V(T), (A(T) \setminus B) \cup B^{\star})$, where $B^{\star} = \{b^{\star} : b \in B\}$. Given an arc $a = (x,y) \in A(T)$, the vertex set $\{x,y\}$ is denoted by $\mathcal{V}(a)$. Similarly, for $B \subseteq A(T)$, the vertex set $\displaystyle\bigcup_{b \in B} \mathcal{V}(b)$ is denoted by $\mathcal{V}(B)$.

Let us say that an arc $a$ is a {\it $v$-arc} when $v \in \mathcal{V}(a)$.
In Theorem~\ref{thmerdos}, the indecomposable tournament $T'$ is obtained from $T$ by adding the vertex $v$. 
It is easily seen that, equivalently, $T'$ is also obtained from any tournament $U$ such that $V(U) = V(T) \cup \{v\}$ and $U-v = T$, by a succession of $v$-arc-reversal operations. Thus Theorem~\ref{thmerdos} is equivalent to the following one.

\begin{thm} \label{thmerdos2}
Let $T$ be a tournament of order at least 5. For every $v \in V(T)$ such that $T-v$ is not transitive of odd order, there exists a set $B$ of $v$-arcs of $T$ such that ${\rm Inv}(T,B)$ is indecomposable.   
\end{thm}

Question $(Q)$ has been investigated by us \cite{Index, Index2} when an allowed operation consists of an arbitrary arc-reversal operation. In \cite{Index2}, we  proved that every tournament $T$ of order at least 5 can be made indecomposable by reversing no more than $\left\lceil \frac{v(T)+1}{4} \right\rceil$ arcs, and that this bound is best possible because, for example, transitive tournaments needs this many arcs to be reversed. 
Formally, given a tournament $T$ of order at least 5,
the {\it decomposability arc-index} (called decomposability index in \cite{Index, Index2}) of $T$, denoted by $\delta(T)$,  is the smallest integer $m$ for which the tournament $T$ can be made indecomposable by reversing a set of $m$ arcs of $T$.
The decomposability arc-index is closely related to another index based on co-modules and called co-modular index. 
\begin{Convention} \normalfont 
	Let $T$ be a tournament. For $X \subseteq V(T)$, $\overline{X}$ denotes $V(T) \setminus X$.
\end{Convention}
The notion of co-module and related notions were introduced in \cite{Index2} as follows. 
Given a tournament $T$, a {\it co-module} of $T$ is a subset $M$ of $V(T)$ such that $M$ or $\overline{M}$ is a nontrivial module of $T$.
A {\it co-modular decomposition} of the tournament $T$ is a set of pairwise disjoint co-modules of $T$. A {\it $\Delta$-decomposition} of $T$ is a co-modular decomposition of $T$ which is of maximum size. Such a size is called the {\it co-modular index} of $T$, and is denoted by $\Delta(T)$. Notice that when $T$ is decomposable, we have $\Delta(T) \geq 2$.
For a nonnegative integer $n$, we denote by $\Delta(n)$ (resp. $\delta(n)$ when $n \geq  5$) the maximum of $\Delta(T)$ (resp. $\delta(T)$) over the tournaments $T$ of order $n$. The relationship between the above two indices is given by the following theorem obtained in \cite{Index2}.
\begin{thm} [\cite{Index2}] \label{deltan} 
For every tournament $T$ of order at least $5$, we have $\delta(T) = \left\lceil \frac{\Delta(T)}{2} \right\rceil$. Moreover, for every integer $n \geq 5$, we have $\Delta(n) = \left\lceil \frac{n+1}{2} \right\rceil$ and,  consequently, $\delta(n) = \delta(\underline{n})= \left\lceil \frac{\Delta(n)}{2} \right\rceil = \left\lceil \frac{n+1}{4} \right\rceil$.
\end{thm}
In this paper, we are interested in Question $(Q)$ by searching for an analogue of Theorem \ref{deltan} for another type of allowed operations motivated by Theorem~\ref{thmerdos2}, and leading to a new index for which we establish the relation with those of Theorem \ref{deltan}. This operation type, that we call {\it subtournament-reversal} operation, consists of reversing simultaneously all the arcs of a subtournament. It was considered in \cite{BBBP}, where the authors study the minimum number of such operations that must be successively applied to a tournament $T$ to make it transitive. Obviously subtournament-reversal operations include arc-reversal ones. For simplicity, given a tournament $T$ and a subset $X$ of $V(T)$, we write ${\rm Inv}(T,X)$ for ${\rm Inv}(T,A(T[X]))$. For example, the {\it dual} tournament of $T$ is the tournament  $T^{\star}={\rm Inv}(T, V(T)) = {\rm Inv}(T, A(T))$. Notice that $T$ and $T^{\star}$ share the same modules. In particular, $T$ is indecomposable if and only if $T^{\star}$ is. 

Let $T$ be a tournament of order at least $5$ which is not transitive of even order. Clearly there is a vertex $v \in V(T) $ such that $T-v$ is not transitive of odd order. By Theorem~\ref{thmerdos2}, there exists a subset $B$ of $v$-arcs of $T$ such that ${\rm Inv}(T,B)$ is indecomposable. In terms of subtournament-reversal operations, this yields the following. Since ${\rm Inv}(T,B) = {\rm Inv}({\rm Inv}(T, \mathcal{V}(B)), \mathcal{V}(B) \setminus \{v\})$, then $T$ can be made indecomposable by at most two successive subtournament-reversal operations. This naturally leads to the following question: can the tournament $T$ be made indecomposable by a single subtournament-reversal operation ? We will prove that the answer is yes, i.e., there exists a subset $X$ of $V(T)$ such that ${\rm Inv}(T, X)$ is indecomposable. This result suggests considering a new decomposability index for tournaments of orders at least $5$, other than transitive ones of even orders. Let $T$ be such a tournament. The {\it decomposability subtournament-index} of $T$, which we denote by $\delta'(T)$, is the smallest integer $m$ for which there exists $X \in \binom{V(T)}{m}$ such that ${\rm Inv}(T, X)$ is indecomposable.
For every integer $n \geq 5$, we denote by $\delta'(n)$ the maximum of $\delta'(T)$ over the tournaments $T$ of order $n$, which are non-transitive when $n$ is even. Theorem~\ref{thm principal} below is the main result of the paper.    
 
 \begin{Notation} \normalfont
 	Let $T$ be a tournament. A {\it minimal co-module} of $T$ is a co-module $M$ of $T$ that is minimal in the set of co-modules of $T$ ordered by inclusion. The set of minimal co-modules of $T$ is denoted by $\text{mc}(T)$.
 \end{Notation}
 
 \begin{thm} \label{thm principal}
 	Given a tournament $T$ of order at least $5$ that is not transitive of even order, there exists $X \in \binom{\cup{\rm mc}(T)}{\Delta(T)}$ such that ${\rm Inv}(T, X)$ is indecomposable. Moreover, we have $\delta'(T) = \Delta(T)$. 
 	\end{thm} 
The next result (see Corollary~\ref{cor principal} below) is an immediate consequence of Theorems~\ref{deltan} and \ref{thm principal} and the following two simple facts, for which we omit the proofs.

\begin{Fact} We have $\delta'(6)=3$.
\end{Fact}

\begin{Fact}
	Given an even integer $n \geq 8$, there exist non-transitive tournaments $T$ of order $n$ such that $\Delta(T) = \Delta(n) = \frac{n}{2}+1$.
\end{Fact}

\begin{cor} \label{cor principal}
 Given a tournament $T$ of order at least $5$ that is not transitive of even order, we have $\delta(T) =\left\lceil \frac{\delta'(T)}{2} \right\rceil = \left\lceil \frac{\Delta(T)}{2} \right\rceil$. Moreover, for every integer $n\geq 5$, we have 
 
 \begin{equation*}
 \delta'(n)= \
 \begin{cases}
  \left\lceil \frac{n+1}{2} \right\rceil \ \ \ \ \text{if} \ \  n \neq 6,\\
 
 \ \ \ 3 \ \ \ \ \ \ \   \text{if} \ \ n=6.
 \end{cases}
 \end{equation*} 
 \end{cor}
The rest of the paper is organized as follows. Section~2 contains required results on modules and co-modules. In Section~3, we prove an important tool for the proof of Theorem~\ref{thm principal} (see Theorem~\ref{thhypergraphe}), which consists of a strengthening of a theorem of Erd\H{o}s et al.~\cite[Theorem 3]{EFHM} about the family of nontrivial modules of a tournament. In Section 4, we prove a finer version of Theorem~\ref{thm principal} (see Theorem~\ref{thm principal2}). Some additional remarks are discussed in Section~5. 

\section{Some known results on modules and co-modules}
We now review some useful properties of the modules and co-modules of a tournament. We begin by the following well-known properties of modules. 

\begin{prop} \label{propo modules}
Let $T$ be a tournament.
\begin{enumerate}
\item Given a subset $W$ of $V(T)$, if $M$ is a module of $T$, then $M \cap W$ is a module of $T[W]$.
\item Given a module $M$ of $T$, if $N$ is a module of $T[M]$, then $N$ is also a module of $T$.
 \item If $M$ and $N$ are modules of $T$, then $M \cap N$ is also a module of $T$.
\end{enumerate}
\end{prop}
We next review some basic properties of co-modules and co-modular decompositions.
\begin{lem} [\cite{Index2}] \label{comod part}
 Given a decomposable tournament $T$, consider a co-modular decomposition $D$ of $T$. The following assertions are satisfied.
 \begin{enumerate}
 \item The tournament $T$ admits at most two singletons which are co-modules of $T$. In particular, $D$ contains at most two singletons. 
  \item If $D$ contains an element $M$ which is not a module of $T$, then the elements of $D \setminus \{M\}$ are nontrivial modules of $T$.
  \item If $D$ is a $\Delta$-decomposition of $T$ and $v(T) \geq 4$, then $D$ contains a nontrivial module of $T$.
 \end{enumerate}
\end{lem}
To continue, we need the following notions and notations introduced in \cite{Index2}.
Let us say that two sets $E$ and $F$ {\it overlap} when $E \cap F \neq \varnothing$, $E \setminus F \neq \varnothing$ and $F \setminus E \neq \varnothing$. By minimality, the elements of $\text{mc}(T)$ are  pairwise incomparable by inclusion, i.e.,  
 \begin{equation} \label{eq overl disj}
 \text{for every } M \neq N \in \text{mc}(T), \text{ either } M \cap N = \varnothing \text{ or } M \text{ and } N \text{ overlap}.
 \end{equation}
 Let $M \in \text{mc}(T)$. We denote by $O_T(M)$ the set of the elements $N \in \text{mc}(T)$ that overlap $M$, i.e., $O_T(M) = \{N \in \text{mc}(T): N \ \text{overlaps } \ M\}$. We set $o_T(M) = |O_T(M)|$. A {\it transitive module} of $T$ is a module $M$ of $T$ such that the subtournament $T[M]$ is transitive. A {\it transitive component} of $T$ is a transitive module of $T$ which is maximal in the set of transitive modules of $T$ ordered by inclusion. A {\it twin} of the tournament $T$ is a module of cardinality 2 of $T$. 
 \begin{Notation} \normalfont
 Given a tournament $T$, the set of twins of $T$ is denoted by $\text{tw}(T)$.
 \end{Notation}
 
 \begin{lem} [\cite{Index2}] \label{trans partition}
 	Given a tournament $T$, the transitive components of $T$ form a partition of $V(T)$.
 \end{lem}
 
 \begin{lem} [\cite{Index2}] \label{degre G_T}
Let $T$ be a tournament and let $M \in {\rm mc}(T)$. 
 We have $o_T(M) \leq 2$. Moreover, if $M \notin {\rm tw}(T)$, then $o_T(M) = 0$.
  \end{lem}

As a consequence of (\ref{eq overl disj}) and Lemma~\ref{degre G_T}, we obtain the following.

\begin{cor} \label{cor not twin disj}
Let $T$ be a tournament, and let $M \in {\rm mc}(T)$ such that $o_T(M) = 0$. We have 
\begin{equation} \label{eq cor}
M \cap N = \varnothing \text{ for every } N \in {\rm mc}(T) \setminus \{M\}.
\end{equation} 
In particular, {\rm (\ref{eq cor})} holds when $M \in {\rm mc}(T) \setminus {\rm tw}(T)$.
\end{cor}

\begin{Notation} \normalfont \label{notat C(k)}
Let $T$ be a tournament with at least three vertices. Suppose that $T$ admits a transitive component $C$ such that $|C|=n \geq 2$. The elements of $C$ can be indexed as $v_0, \ldots, v_{n-1}$ in such a way that $T[C] = (C, \{(v_i,v_j): 0 \leq i < j \leq n-1\})$. For every $k \in \{0, \ldots, n-2\}$, the pair $\{v_{k}, v_{k+1}\}$ is a twin of $T[C]$  and thus of $T$ (see Assertion~2 of Proposition~\ref{propo modules}). The unique element of ${\rm mc}(T)$ that is contained in $\{v_{k}, v_{k+1}\}$  is denoted by $C(k)$. When $n \geq 3$, for every $k \in \{0, \ldots, n-2\}$, we have (e.g., see \cite{Index2})
\begin{equation*}
C(k) = \
\begin{cases}
\{v_0\} \ \text{or} \ \{v_0, v_1\} \ \ \ \ \ \ \ \ \ \ \ \ \  \text{if} \ \  k=0,\\

\{v_{n-1}\} \ \text{or} \ \{v_{n-2},v_{n-1}\} \  \ \ \ \  \text{if} \ k= n-2,\\        

\{v_k, v_{k+1}\} \ \ \ \ \ \ \ \ \ \ \ \ \ \ \ \ \ \ \ \ \hspace{0.09cm} \text{otherwise.}
\end{cases}
\end{equation*}
\end{Notation}

\begin{lem} [\cite{Index2}] \label{lem C mnc}
	Let $T$ be a tournament of order at least $3$. Suppose that $T$ admits a transitive component $C$ such that $|C|=n \geq 2$. Given $M \subseteq V(T)$, the following assertions are equivalent.
	\begin{enumerate}
		\item $M \in {\rm mc}(T)$ and $M \cap C \neq \varnothing$.
		\item $M \in \{C(0), \ldots, C(n-2)\}$.
	\end{enumerate}
\end{lem}
\begin{Notation} \normalfont
Given a tournament $T$, the set of nontrivial modules of $T$ is denoted by $\mathcal{M}(T)$.
\end{Notation}
 \section{Bipartitions of $\mathcal{M}(T)$ by transversals of $\text{mc}(T)$}
We need some terminology from hypergraphs (e.g., see \cite{CBerge}) that we introduce in terms of families of sets.
We only have to consider finite families of finite sets.
Let $\mathcal{ F}$ be such a family. 
A {\it stable} set of $\mathcal{ F}$ is any (finite) set that contains no element of $\mathcal{ F}$.
The family $\mathcal{ F}$ is {\it bipartite} (or has {\it Bernstein property} \cite{EFHM}) if there exists a set $S$ such that both $S$ and $(\cup \mathcal{F}) \setminus S$ are stable sets of $\mathcal{ F}$. 
When we say that $\mathcal{ F}$ is bipartite by $S$,  we mean that $S$ and $(\cup \mathcal{F}) \setminus S$ are stable sets of $\mathcal{ F}$.
Erd\H{o}s et al. \cite{EFHM} obtained the following theorem. 
\begin{thm}[\cite{EFHM}] \label{thm erdos modules}
	For any tournament $T$, $\mathcal{M}(T)$ is bipartite.
\end{thm}
We need a strengthened version of Theorem~\ref{thm erdos modules} (see Theorem~\ref{thhypergraphe} below). In that order, we continue with the following notions about finite families of finite sets. To begin, we strengthen the notion of bipartite family to that of strictly bipartite family as follows. The family $\mathcal{F}$ is {\it strictly bipartite} if there exists a (finite) set $S$ that overlaps every element of $\mathcal{F}$. In this instance, we say that $\mathcal{F}$ is strictly bipartite by $S$. Clearly, if the family $\mathcal{F}$ is strictly bipartite by $S$, then it is bipartite by the same $S$. But a bipartite family is not necessarily strictly bipartite.    
A {\it transversal} of the family $\mathcal{F}$ is any (finite) set $R$ that intersects each element of $\mathcal{F}$, i.e., such that $F \cap R \neq \varnothing$ for every $F \in \mathcal{ F}$. Observe that a set $R$ is a transversal of the family $\mathcal{ F}$ if and only if $(\cup \mathcal{F}) \setminus R$ is a stable set of $\mathcal{F}$. A {\it minimum transversal} of $\mathcal{F}$ is a transversal of $\mathcal{F}$ which is of minimum size. Observe that a minimum transversal of $\mathcal{ F}$ is always contained in $\cup \mathcal{ F}$. The {\it transversal number} of $\mathcal{F}$, denoted by $\tau(\mathcal{F})$, is the size of a minimum transversal of $\mathcal{F}$. An  {\it exact transversal}  of $\mathcal{ F}$ is a transversal $R$ of $\mathcal{F}$ such that $\vert F \cap R\vert =1$ for every $F \in \mathcal{F}$. A {\it matching} of $\mathcal{F}$ is a subset of pairwise disjoint elements of $\mathcal{F}$. The {\it matching number} of $\mathcal{F}$, denoted by $\nu(\mathcal{F})$, is the size of a {\it maximum matching} of $\mathcal{F}$, i.e., a matching of $\mathcal{F}$ which is of maximum size. For example, in \cite{Index2}, a maximum matching of ${\rm mc}(T)$ is called a {\it $\delta$-decomposition} of $T$, where $T$ is a tournament. It is easy to see that
\begin{equation} \label{eq nu = Delta}
\text{for every tournament } T, \text{ we have } \Delta(T) = \nu(\text{mc}(T)).
\end{equation}
 Clearly, the transversal and matching numbers of the family $\mathcal{F}$ satisfy
\begin{equation} \label{eq nu tau}
\nu(\mathcal{F})\leq \tau(\mathcal{F}).
\end{equation}
When equality holds in (\ref{eq nu tau}), i.e. $\nu(\mathcal{F}) = \tau(\mathcal{F})$, we say that $\mathcal{F}$ has the {\it K\"{o}nig property}.
\begin{thm}\label{thhypergraphe}
 For any tournament $T$ that is not transitive of even order, $\mathcal{M}(T)$ is strictly bipartite by an exact and minimum transversal of ${\rm mc}(T)$. Moreover, ${\rm mc}(T)$ has the K\"{o}nig property.
\end{thm}
\begin{proof}
	Let $T$ be a tournament. If $T$ is indecomposable, then the theorem trivially holds because $\mathcal{M}(T) = \text{mc}(T) = \varnothing$. Hence suppose that $T$ is decomposable, and in particular $v(T) \geq 3$. Further suppose that $T$ is not transitive of even order.
	We will construct a transversal $R$ of $\text{mc}(T)$ that is minimum and exact. For this purpose, we have to divide $\text{mc}(T)$ into the disjoint union of $\text{mc}_0(T) = \text{mc}(T) \setminus \text{tw}(T)$ and $\text{mc}_1(T) = \text{mc}(T) \cap \text{tw}(T)$. 
	
	Let $f$ be a choice function on $\text{mc}_0(T)$.
Since the elements of $\text{mc}_0(T)$ are pairwise disjoint by Corollary~\ref{cor not twin disj}, then $f(\text{mc}_0(T))$ is an exact and minimum transversal of $\text{mc}_0(T)$. Set $R_0 = f(\text{mc}_0(T))$. 

We now consider the set $\textbf{C}(T)$ of the transitive components $C$ of $T$ such that $|C| \geq 2$. 
With every $C \in \textbf{C}(T)$, we associate the subset $r(C)$ of $C$ defined as follows. 
\begin{equation} \label{eq C'}
 r(C)=\begin{cases}\{v_{i} \in C: i \ \text{is  even}\} \ \ \ \ \ \ \ \ \ \ \ \ \ \ \text{ if } \{v_{0}\}\in \text{mc}(T),\\
\{v_{\vert C\vert -i} \in C: i \ \text{is  odd}\} \ \ \  \ \ \ \ \ \ \ \ \text{ if } \{v_{\vert C\vert -1}\}\in \text{mc}(T),\\
\{v_{i} \in C: i \ \text{is  odd}\} \ \ \ \ \ \ \  \ \ \ \  \ \ \ \ \text{ otherwise,}
\end{cases}
\end{equation}
where $v_0, \ldots, v_{|C|-1}$ are the elements of $C$ indexed in such a way that $T[C] = (C, \{(v_i,v_j): 0 \leq i < j \leq \vert C\vert-1\})$.
To verify that $r(C)$ is well-defined, we have to show that if $\{v_{0}\}\in \text{mc}(T)$ and $\{v_{|C|-1}\}\in \text{mc}(T)$, then $\{v_{i} \in C: i \ \text{is  even}\} = \{v_{\vert C\vert -i} \in C: i \ \text{is  odd}\}$. 
The following observation follows from the definitions of module and co-module.
\begin{Observation} \label{obs MNN'}
Given a module $M$ of a tournament $U$, if there exist $I \neq J \in \rm{mc}(U)$ such that $|I| = |J|=1$ and $I \cup J \subseteq M$, then $M = V(U)$.
\end{Observation}
 Suppose that $\{v_{0}\}\in \text{mc}(T)$ and $\{v_{|C|-1}\}\in \text{mc}(T)$. Since $C$ is a module of $T$, it follows from Observation~\ref{obs MNN'} that $C = V(T)$, i.e., $T$ is transitive. Since $T$ is not transitive of even order, it follows that $|C|$ is odd. Therefore $\{v_{i} \in C: i \ \text{is  even}\} = \{v_{\vert C\vert -i} \in C: i \ \text{is  odd}\}$. Thus $r(C)$ is well-defined.
 
Set $R_1=\cup \{r(C): C \in \textbf{C}(T)\}$ and $R=R_0 \cup R_1$.
We have to prove that $R$ is an exact and minimum transversal  of $\text{mc}(T)$. 
 
 To begin, we prove that $R$ is an exact transversal of $\text{mc}(T)$. 
 Let $M \in \text{mc}(T)$. First, suppose $M \in \text{mc}_0(T)$. Since $R_0$ is an exact transversal of $\text{mc}_0(T)$, we have $|M \cap R_0| =1$. If $M \cap R_1 = \varnothing$, then $|M \cap R| =1$ as desired. Hence suppose $M \cap R_1 \neq \varnothing$.
 By the definition of $R_1$, there exists $C \in \textbf{C}(T)$ such that $M \cap r(C) \neq \varnothing$ and thus $M \cap C \neq \varnothing$. By Lemma~\ref{lem C mnc}, we have $M = C(k)$ for some $k \in \{0, \ldots, |C|-2\}$. But $C(k) \in \text{tw}(T)$ or $|C(k)|=1$ (see Notation~\ref{notat C(k)}). Since $M \notin \text{tw}(T)$, we obtain $|M| =1$. Thus, $|M \cap R|=1$ as required.  
 Second, suppose $M \in \text{mc}_1(T)$. We will prove that $M \cap R_0 = \varnothing$ and $|M \cap R_1| =1$, which implies $|M \cap R| =1$. Let $C$ be the transitive component of $T$ containing $M$. Since $M \in {\rm tw}(T)$, by the construction of $r(C)$, we have $|M \cap r(C)|=1$. It follows from Lemma~\ref{trans partition} that $|M \cap R_1| =1$. Now suppose to the contrary that $M \cap R_0 \neq \varnothing$. By the definition of $R_0$, there exists $N \in \text{mc}_0(T)$ such that $M \cap N \neq \varnothing$. Since $o_T(N) = 0$ by Lemma~\ref{degre G_T}, $M =N$ by Corollary~\ref{cor not twin disj}, which is not possible because $M \in \text{tw}(T)$ and $N \notin \text{tw}(T)$. Thus $M \cap R_0 = \varnothing$. We conclude that $R$ is an exact transversal of $\text{mc}(T)$.   

We now prove that $\text{mc}(T)$ has the K\"{o}nig property, and that its exact transversal $R$ is minimum. Let $\text{mc}_2(T)$ be a maximum matching of $\text{mc}(T)$. Since $R$ is an exact transversal of ${\rm mc}(T)$ and ${\rm mc}_2(T) \subseteq {\rm mc}(T)$, $R$ is an exact transversal of ${\rm mc}_2(T)$ as well. Since the elements of ${\rm mc}_2(T)$ are pairwise disjoint, to prove that the exact transversal $R$ of ${\rm mc}_2(T)$ is also minimum, it suffices to prove that $R \subseteq \cup \text{mc}_2(T)$. By Corollary~\ref{cor not twin disj}, $M \cap N = \varnothing$ for every $M \in \text{mc}_0(T)$ and $N \in\text{mc}(T) \setminus \{M\}$. Hence, it follows from the maximality of $\text{mc}_2(T)$ that $\text{mc}_0(T) \subseteq \text{mc}_2(T)$. Since $R_0 \subseteq \cup \text{mc}_0(T)$, we obtain $R_0 \subseteq \cup \text{mc}_2(T)$. Now let $C \in \textbf{C}(T)$. We may assume $T[C] = \underline{n}$ for some integer $n \geq 2$. We have to prove that $r(C) \subseteq \cup \text{mc}_2(T)$, which implies $R_1 \subseteq \cup \text{mc}_2(T)$ and thus $R \subseteq \cup \text{mc}_2(T)$, as desired. Suppose to the contrary that $r(C) \setminus \cup \text{mc}_2(T) \neq \varnothing$. Let $k \in r(C) \setminus \cup \text{mc}_2(T)$. It is not difficult to verify that in this case, we have $n \geq 3$. Set $K_0 = \{i \in C : i < k\}$, $K_1 = \{i \in C : i > k\}$, $D_0 = \{C(0), \ldots, C(n-2)\}$, and $D_1 = D_0 \cap \text{mc}_2(T)$. By Lemma~\ref{lem C mnc} and the maximality of $\text{mc}_2(T)$, $D_1$ is a maximum matching of $D_0$. Moreover, since $k \notin \cup D_1$, $D_1$ is the disjoint union of  $ D_1 \cap 2^{K_0}$ and $D_1 \cap 2^{K_1}$ (see Notation~\ref{notat C(k)}). Hence by the maximality of $D_1$, 
\begin{equation} \label{eq K0}
D_1 \cap 2^{K_0}  \text{ is a maximum matching of } D_0 \cap 2^{K_0 \cup \{k\}}.
\end{equation}
 We now distinguish the following two cases. Suppose $\{\{0\}, \{n-1\}\} \cap \text{mc}(T) = \varnothing$. In this instance, we have $D_0= \{\{i,i+1\} : 0 \leq i \leq n-2\}$ (see Notation~\ref{notat C(k)}) and $r(C) = \{i \in C : i \text{ is odd}\}$ (see (\ref{eq C'})). Thus $k$ is odd because $k \in r(C)$. Since the elements of $D_1$ are disjoint and of size $2$, and since $|K_0| =k$, it follows that $|D_1 \cap 2^{K_0}| \leq \frac{k-1}{2}$. On the other hand, $\{\{i,i+1\} \in 2^{K_0 \cup \{k\}} : i \text{ is even}\}$ is a matching of size $\frac{k+1}{2}$ of $D_0 \cap 2^{K_0 \cup \{k\}}$. This contradicts (\ref{eq K0}). Now suppose $\{\{0\}, \{n-1\}\} \cap \text{mc}(T) \neq \varnothing$. By interchanging $T$ and $T^{\star}$, as well as $i$ and $n-1-i$ for $i \in C$, we may assume $\{0\} \in \text{mc}(T)$. 
Thus $r(C) = \{i \in C : i \text{ is even}\}$ (see (\ref{eq C'})), and hence $k$ is even. Since the elements of $(D_1 \cap 2^{K_0}) \setminus \{\{0\}\}$ are disjoint and of size $2$ (see Notation~\ref{notat C(k)}), it follows that $|D_1 \cap 2^{K_0}| \leq \frac{k}{2}$. On the other hand, if $\{k\} \in \text{mc}(T)$, then since $k \notin \cup{\rm mc}_2(T)$, ${\rm mc}_2(T) \cup \{\{k\}\}$ is a matching of ${\rm mc}(T)$, which contradicts the maximality of ${\rm mc}_2(T)$. Thus $\{k\} \notin \text{mc}(T)$. Therefore, $\{\{0\}\} \cup (\{\{i,i+1\} \in 2^{K_0 \cup \{k\}} : i \text{ is odd}\})$ is a matching of $D_0 \cap 2^{K_0 \cup \{k\}}$. Since this matching is of size $1 + \frac{k}{2}$ and $|D_1 \cap 2^{K_0}| \leq \frac{k}{2}$, we again contradict (\ref{eq K0}). We conclude that $R \subseteq \cup \text{mc}_2(T)$, as claimed. Therefore, $R$ is a minimum transversal of $\text{mc}_2(T)$. Since $R$ is a transversal of ${\rm mc}(T)$ and ${\rm mc}_2(T) \subseteq {\rm mc}(T)$, it follows that $R$ is also a minimum transveral of  $\text{mc}(T)$, as required. Moreover, we have $\vert R\vert =\vert \text{mc}_2(T)\vert = \tau(\text{mc}(T))$. Since $\vert \text{mc}_2(T)\vert = \nu(\text{mc}(T))$, we obtain  $\tau(\text{mc}(T)) = \nu(\text{mc}(T))$, i.e.,  $\text{mc}(T)$ has the K\"{o}nig property.

Finally, to prove that  $\mathcal{ M}(T)$ is strictly bipartite by $R$, we have to prove that $M\cap R\neq \varnothing$, $\overline{M} \cap R \neq \varnothing$, $M\cap \overline{R}\neq \varnothing$ for every $M\in \mathcal{M}(T)$. Let $M\in \mathcal{M}(T)$. Since $M$ and $\overline{M}$ are co-modules of $T$, and $R$ is a transversal of ${\rm mc}(T)$, we have $M \cap R \neq \varnothing$ and $\overline{M} \cap R \neq \varnothing$.
  Suppose to the contrary that $M\subseteq R$. There exists a subset $I$ of $M$ such that $I \in \text{mc}(T)$. Since $I \subseteq R$ and $|I \cap R|=1$ because $R$ is an exact transversal of $\text{mc}(T)$, we have $\vert I\vert =1$. 
 By Assertion 3 of Proposition \ref{propo modules}, $M\cap \overline{I}$ is a module of $T$ because $\overline{I}$ and $M$ are. Suppose $\vert M\cap \overline{I}\vert \geq 2$. Since $M\cap \overline{I}$ is a nontrivial module of $T$, $M\cap \overline{I}$ contains an element $J$ of ${\rm mc}(T)$. By reasoning as above, we obtain $\vert J\vert =1$. It follows from Observation~\ref{obs MNN'} that $M=V(T)$, a contradiction. Thus $\vert M\cap \overline{I}\vert =1$. Since $\vert M \cap I\vert =1$, it follows that $M$ is a twin of $T$. Let $C$ be the element of $\textbf{C}(T)$ containing $M$. We may assume that $T[C] = \underline{n}$ for some integer $n \geq 2$.
   By interchanging $T$ and $T^{\star}$, we may assume that $I=\{0\}$ and $M=\{0,1\}$. Thus $r(C)=\{i \in C: i \text{ is even}\}$ (see (\ref{eq C'})). Since $R\cap C=r(C)$, it follows that $1 \in M \setminus R$, which contradicts $M\subseteq R$. This completes the proof.
 \end{proof}

Alternatively, Theorem~\ref{thhypergraphe} can be stated as in Theorem~\ref{not tr} below.

\begin{Notation} \label{not tr0} \normalfont
	Given a tournament $T$, we denote by $\text{tr}(T)$ the set of exact and minimum transversals $R$ of ${\rm mc}(T)$ such that $\mathcal{M}(T)$ is strictly bipartite by $R$. 
\end{Notation}	

\begin{thm} \label{not tr}
For any tournament $T$ that is not transitive of even order, we have $\varnothing \neq {\rm tr}(T) \subseteq \binom{\cup {\rm mc}(T)}{\Delta(T)}$.	
	\end{thm}	
\begin{proof}
	Let $T$ be a tournament that is not transitive of even order. The first assertion of Theorem~\ref{thhypergraphe} says that ${\rm tr}(T) \neq \varnothing$. Let $R \in {\rm tr}(T)$. By minimality of the transversal $R$ of ${\rm mc}(T)$, we have $R \subseteq \cup {\rm mc}(T)$. Since ${\rm mc}(T)$ has the K\"{o}nig property (see Theorem~\ref{thhypergraphe}) and $\Delta(T) = \nu({\rm mc}(T))$ (see (\ref{eq nu = Delta})), we obtain $|R| = \Delta(T)$. Thus $R \in \binom{\cup {\rm mc}(T)}{\Delta(T)}$, and hence ${\rm tr}(T) \subseteq \binom{\cup {\rm mc}(T)}{\Delta(T)}$.	
\end{proof} 
 The proof of Theorem~\ref{not tr} shows how Theorem~\ref{thhypergraphe} implies Theorem~\ref{not tr}. Conversely, it is easily seen that Theorem~\ref{not tr} immediately implies Theorem~\ref{thhypergraphe}. Thus, Theorems~\ref{thhypergraphe} and \ref{not tr} are equivalent.   
 \section{Proof of Theorem \ref{thm principal}}
 Since ${\rm tr}(T) \subseteq \binom{{\rm mc}(T)}{\Delta(T)}$ in Theorem~\ref{not tr}, then Theorem~\ref{thm principal} is a direct consequence of Theorem~\ref{thm principal2} below, which gives a finer localization of the subset $X$ such that ${\rm Inv}(T, X)$ is indecomposable. 
 \begin{thm} \label{thm principal2}
 	Given a tournament $T$ of order at least $5$ that is not transitive of even order, there exists $X \in {\rm tr}(T)$ such that ${\rm Inv}(T, X)$ is indecomposable. Moreover, we have $\delta'(T) = \Delta(T)$. 
 \end{thm}
The aim of this section is to prove Theorem~\ref{thm principal2}. We first introduce some convenient notations. Let $T$ be a tournament. We also denote by $T$ the function $T : (V(T)\times V(T)) \setminus \{(x,x) : x \in V(T)\} \longrightarrow \mathbb{Z}_2$, defined by  
 \begin{equation*}
 	T(x,y)= \
 	\begin{cases}
 		1 \ \ \ \ \text{if} \ \  (x,y) \in A(T),\\
 		
 		0 \ \ \ \   \text{if} \ \ (x,y) \notin A(T).
 	\end{cases}
 \end{equation*}
For instance, given a subset $R$ of $V(T)$, we have

\begin{equation} \label{1+}
({\rm Inv}(T, R))(x,y)= \
\begin{cases}
1 + T(x,y) \ \ \ \ \text{if }  \{x,y\} \subseteq R,\\

T(x,y) \ \ \ \ \ \ \ \   \text{ otherwise}.
\end{cases}
\end{equation}
For two disjoint subsets $X$ and $Y$ of $V(T)$, we write $X \equiv_T Y$ to mean that $T(x,y) = T (x',y')$ for every $x, x' \in X$ and $y, y' \in Y$. 
 For more precision, we write $T(X,Y) =1$ (resp. $T(X,Y) =0$) to mean that $T(x,y) =1$ (resp. $T(x,y) =0$) for every $x \in X$ and $y \in Y$. When $X = \{x\}$ for some $x \in V(T)$, we write $x \equiv_T Y$ for $\{x\} \equiv_T Y$, and $T(x,Y)$ for $T(\{x\}, Y)$. Notice that for $M \subseteq V(T)$, $M$ is a module of $T$ if and only if $x \equiv_T M$ for every $x \in \overline{M}$, or, equivalently, if and only if $T(x,u) = T(x,v)$ for every $x \in \overline{M}$ and $u,v \in M$.
 
 The following facts come from a simple examination of the adjacency relation in ${\rm Inv}(T,R)$ compared to that in $T$, as observed in (\ref{1+}).
 \begin{Fact}\label{rinv} 
  	Let $T$ be a tournament and $M, X \subseteq V(T)$.
  	\begin{enumerate}
  	 \item If $M$ and $X$ do not overlap, then $M$ is a module of ${\rm Inv}(T,X)$ if and only if it is a module of $T$.
  	 \item If $M$ and $X$ overlap, then  $M \notin \mathcal{M}(T) \cap \mathcal{M}({\rm Inv}(T,X))$.
  	 \item If ${\rm Inv}(T,X)$ is indecomposable, then $X$ is a transversal of ${\rm mc}(T)$.
  	 \item If ${\rm Inv}(T,X)$ is indecomposable, then $\mathcal{M}(T)$ is strictly bipartite by $X$.
  	 \end{enumerate}
  \end{Fact}
Fact~\ref{rinv} shows how Theorems~\ref{thhypergraphe} and \ref{thm principal} are closely related (see also Theorem~\ref{pinv} below). More details about their relationship are provided in Section~5. In the rest of this section, Theorem~\ref{thhypergraphe} will play a central role in the proof of Theorem~\ref{thm principal2}.

\begin{lem}\label{linv}
Let $T$ be a tournament that is not transitive of even order, let $R\in {\rm tr}(T)$, and set $T'={\rm Inv}(T,R)$. The following  assertions hold.
\begin{enumerate}
	\item We have $\mathcal{M}(T) \cap \mathcal{M}(T') = \varnothing$.
\item  $\mathcal{M}(T')$ is strictly bipartite by $R$.
\item Every nontrivial module of $T'$ is a transversal of ${\rm mc}(T)$.  
\end{enumerate} 
\end{lem}
\begin{proof} We clearly can suppose that $T$ and $T'$ are decomposable.
	Since $\mathcal{M}(T)$ is strictly bipartite by $R$, the first assertion follows from Assertion~2 of Fact~\ref{rinv}.
 Let $M\in \mathcal{M}(T')$.  If $M$ is a module of $T$, then $R$ overlaps $M$ because $\mathcal{M}(T)$ is strictly bipartite by $R$. If $M$ is not a module of $T$, then $R$ overlaps $M$ by Assertion~1 of Fact~\ref{rinv}. Thus the second assertion holds.
 For the third assertion, let $I \in {\rm mc}(T)$. We have to prove that $M \cap I \neq \varnothing$. Since $R$ and $M$ overlap by the second assertion, and since $R$ is  a transversal of $\text{mc}(T)$, we have $M\cap R \neq \varnothing$, $M\cap \overline{R} \neq \varnothing$, and $I \cap R \neq \varnothing$. So let $x\in M\cap R$, $y\in M\cap \overline{R}$, and $u \in I \cap R$. Suppose to the contrary that $M\cap I=\varnothing$. 
 If $\overline{I}$ is a  module of $T$, then $u \equiv_T \{x,y\}$ and thus $u \not\equiv_{T'} \{x,y\}$, which contradicts that $M$ is a module of $T'$. Therefore $I \in \mathcal{M}(T)$. Since $\mathcal{M}(T)$ is bipartite by $R$, we have $I \cap \overline{R} \neq \varnothing$. So let $v \in I\cap \overline{R}$. Observe that $x$, $y$, $u$, and $v$ are pairwise distinct. Since $I$ is a module of $T$, then either $\{x,y\} \equiv_T I$ and thus $u \not\equiv_{T'} M$, or $T(x,I) = 1 + T(y,I)$ and thus $v \not\equiv_{T'} M$. In both cases, this contradicts $M \in \mathcal{M}(T')$.  
\end{proof}

\begin{Notation} \normalfont
	In the proof of Theorem~\ref{pinv} below, we need to divide ${\rm mc}(T)$ into the disjoint union of ${\rm mc}^+(T) = {\rm mc}(T) \cap \mathcal{M}(T)$ and ${\rm mc}^-(T) = {\rm mc}(T) \setminus \mathcal{M}(T)$.
	\end{Notation}
	 Observe that by Assertions~1 and 2 of Lemma~\ref{comod part}, we have $|{\rm mc}^-(T)| \leq 2$. Moreover, it follows from Assertion~3 of Lemma~\ref{comod part} that 
	\begin{equation} \label{eq mc+}	
	\text{if the tournament } T \text{ is decomposable and } v(T) \geq 4, \text{ then } {\rm mc}^+(T) \neq \varnothing.  
	\end{equation}

 \begin{thm}\label{pinv}
 Let $T$ be a tournament  that  is not transitive of even order.
 If $\Delta(T)\geq 3$, then for every $R\in {\rm tr}(T)$, the tournament ${\rm Inv}(T,R)$ is indecomposable.
 \end{thm}
 \begin{proof}
 Suppose $\Delta(T)\geq 3$. Let $R\in {\rm tr}(T)$ and set $T'={\rm Inv}(T,R)$. We first prove the following claim, which is a  strengthening of Assertion~3 of Lemma~\ref{linv} under the hypothesis $\Delta(T)\geq 3$.
 \begin{claim} \label{claim}
 Every nontrivial module of $T'$ is an exact transversal of ${\rm mc}(T)$. 
 \end{claim}	
 \begin{proof}[Proof of Claim~\ref{claim}]
 	Let $M\in \mathcal{M}(T')$. To prove that  $M$ is an exact transversal of $\text{mc}(T)$,
 	 we first show that
 	\begin{numcases}{\text{for every } I\in \mathcal{M}(T), \text{ if } \vert M\cap I\vert \geq 2, \text{ then }}
 M \cap I \text{ and } R \text{ overlap}, \label{q}\\ 
\text{and } M \cup I = V(T). \label{qq} 
 	\end{numcases}
 Let $I\in \mathcal{M}(T)$ such that $\vert M\cap I\vert \geq 2$. 
 	By Assertion~1 of Proposition~\ref{propo modules}, $M \cap I$ is a module of $T[M]$. Suppose to the contrary that (\ref{q}) does not hold. In this instance, $M \cap I$ and $M \cap R$ do not overlap. Therefore, since $T'[M] = {\rm Inv}(T[M], M \cap R)$, it follows from Assertion~1 of Fact~\ref{rinv} applied to the tournament $T[M]$ that $M \cap I$ is also a module of $T'[M]$. Thus,
 $M \cap I \in \mathcal{M}(T')$ by Assertion~2 of Proposition~\ref{propo modules}. Since (\ref{q}) does not hold, this contradicts Assertion~2 of Lemma~\ref{linv}. Thus (\ref{q}) holds. Hence, if $R \cap \overline{M}\cap \overline{I} \neq \varnothing$, then for $x \in R \cap \overline{M}\cap \overline{I}$, since $|M \cap I| \geq 2$ and $x \equiv_T M \cap I$ because $I \in \mathcal{M}(T)$, we have $x \not\equiv_{T'} M \cap I$, which contradicts $M\in \mathcal{M}(T')$. Thus $R\subseteq M \cup I$. Since $M \in \mathcal{M}(T')$, $I \in \mathcal{M}(T)$, $M \cap I \neq \varnothing$, and $R\subseteq M \cup I$, we obtain that $M \cup I$ is a module of both $T$ and $T'$. It follows from Assertion~1 of Lemma~\ref{linv} that the module $M \cup I$ is trivial. Since $|M \cap I| \geq 2$, (\ref{qq}) holds. 
 	
 	Now let $I \in {\rm mc}(T)$. To prove that $|M \cap I| = 1$, we distinguish the following two cases.
 	\begin{itemize}
 	\item  Suppose $I \in {\rm mc}^+(T)$. Suppose to the contrary that $|M \cap I| \geq 2$. Since $\overline{M} \cap R \neq \varnothing$ by Assertion~2 of Lemma~\ref{linv}, it follows from (\ref{qq}) that $\overline{M} \cap R \cap I \neq \varnothing$. Therefore, since $|I \cap R| =1$ because $R$ is an exact transversal of ${\rm mc}(T)$, we obtain $M \cap I \cap R = \varnothing$, which contradicts (\ref{q}). Thus $|M \cap I| \leq 1$, and hence $|M \cap I| = 1$ by Assertion~3 of Lemma~\ref{linv}. 
 	
 	\item Suppose $I \in {\rm mc}^-(T)$. By Assertion~3 of Lemma~\ref{linv}, it suffices to prove that $|I| =1$. Suppose $|I| \neq 1$. Since $I \in {\rm mc}^{-}(T)$, we have $\overline{I}\in \mathcal{M}(T)$ and 
 	\begin{equation} \label{eqI}
 	I \text{ is not a module of } T.
 	\end{equation}
 	Moreover, since $\Delta(T) = \nu({\rm mc}(T))\geq 3$ by hypothesis, and $o_{T}(I)=0$ by Lemma~ \ref{degre G_T}, then $\vert M \cap \overline{I} \vert\geq 2$ by Assertion~3 of Lemma~\ref{linv}. It follows from (\ref{qq}) that $M \cup \overline{I} = V(T)$, i.e. $I \subseteq M$.
Moreover, $M \cap \overline{I}$ and $R$ overlap by (\ref{q}). Let $x \in M \cap \overline{I} \cap R$. Since $R\subseteq \cup \text{mc}(T)$ (see Theorem~\ref{not tr}), then $x \in J$ for some $J \in {\rm mc}(T)$. More precisely, $J \in {\rm mc}(T) \setminus \{I\}$ because $x \in J \cap \overline{I}$, and hence $I \cap J = \varnothing$ because $o_{T}(I)=0$. In addition, since $I$ is not a module of $T$ (see (\ref{eqI})) and $I \cap J = \varnothing$, then $J \in {\rm mc}^+(T)$ by Assertion~2 of Lemma~\ref{comod part}. It follows from the first case that $M \cap J = \{x\}$. We also have $R \cap J=\{x\}$ because $R$ is an exact transversal of $\text{mc}(T)$. Moreover, $J$ and $R$ overlap because $J \in \mathcal{M}(T)$ and $\mathcal{M}(T)$ is strictly bipartite by $R$. So consider a vertex $y\in J \cap \overline{R}$. Since $M \cap J = R \cap J =\{x\}$, we have $y \in \overline{M} \cap \overline{R}$. Moreover, since $M$ is a module of $T'$, by interchanging $T$ and $T^{\star}$, we may assume that $T'(y,M)=1$ and thus $T(y,M) =1$. In particular,  $T(y,I) =1$ because $I \subseteq M$. Since $\overline{I}$ is a module of $T$, it follows that $T(\overline{I}, I) = 1$, which contradicts (\ref{eqI}). \qedhere
\end{itemize}
\end{proof}
Suppose to the contrary that $T'$ is decomposable. By (\ref{eq mc+}), we have $\text{mc}^{+}(T') \neq \varnothing$. Let $M \in \text{mc}^{+}(T')$. 
\begin{claim} \label{claim2}
We have $\vert M\cap \overline{R}\vert =1$.	
\end{claim}	
\begin{proof}[Proof of Claim~\ref{claim2}]
 Suppose $\vert M\cap \overline{R}\vert \neq 1$. Since $M$ and $R$ overlap by Assertion~2 of Lemma~\ref{linv}, we have $2 \leq \vert M\cap \overline{R}\vert \leq v(T)-2$. It follows from the minimality of $M$ as a co-module of $T'$ that $M \cap \overline{R}$ is not a module of $T'$. Since $M$ is a module of $T'$ and $M \cap \overline{R}$ is not, there exist $x \neq y \in M \cap \overline{R}$ and $u \in M \cap R$ such that $T'(u,x) \neq T'(u,y)$, and thus 
 \begin{equation} \label{qu}
 T(u,x) \neq T(u,y).
 \end{equation}
 Since $R \subseteq \cup {\rm mc}(T)$(see Theorem~\ref{not tr}), we have $u\in I$ for some $I\in \text{mc}(T)$. Since $M$ is an exact transversal of ${\rm mc}(T)$ by Claim~\ref{claim}, we have $I\cap M=\{u\}$. If $I = \{u\}$, then $\overline{\{u\}}$ is a module of $T$, which contradicts (\ref{qu}). Thus $I \neq \{u\}$. Since $I\cap M=\{u\}$ and $I \neq \{u\}$, we have $I \cap \overline{M} \neq \varnothing$. Let $v \in I \cap \overline{M}$. 
 Since $M$ is a module of $T$, we have $T'(v,x) = T'(v,y)$, and thus 
 \begin{equation} \label{qv}
 T(v,x) = T(v,y).
 \end{equation} 
 Since $\{u,v\} \subseteq I$ and $\{x,y\} \subseteq \overline{I}$, it follows from (\ref{qu}) and (\ref{qv}) that $I$ is not a module of $T$. Therefore, $\overline{I}$ is a module of $T$, contradicting (\ref{qu}). Thus $|M\cap \overline{R}| =1$.     
 \end{proof}  
 Now since $\Delta(T) = \nu({\rm mc}(T)) \geq 3$ and $M$ is a transversal of ${\rm mc}(T)$ by Claim~\ref{claim}, we have $|M| \geq 3$. Since $|M\cap \overline{R}| =1$ by Claim~\ref{claim2}, we obtain $|M\cap R| \geq 2$. Let $u$ be the element 
 of $M\cap \overline{R}$. Recall that $M$ and $R$ overlap by Assertion~2 of Lemma~\ref{linv}. So let $v \in R \cap \overline{M}$. Since $R \subseteq \cup {\rm mc}(T)$ (see Theorem~\ref{not tr}), $v \in I$ for some $I \in {\rm mc}(T)$. Since $M$ is an exact transversal of ${\rm mc}(T)$ by Claim~\ref{claim}, we have $|I \cap M|=1$. More precisely, $I \cap M = \{u\}$ because $R$ is also an exact transversal of ${\rm mc}(T)$. 
 If $M \setminus I$ is a module of $T'[M]$, then $M \setminus I$ is a nontrivial module of $T'$ (see Assertion~2 of Proposition~\ref{propo modules}), which contradicts the minimality of $M$ as a co-module of $T'$. Thus, $M\setminus I$ is not a module of $T'$. Therefore, there exist $x \neq y \in M \cap R$ such that $T'(u,x) \neq T'(u,y)$ and thus $T(u,x) \neq T(u,y)$. In particular, $\overline{I}$ is not a module of $T$. Since $I \in {\rm mc}(T)$, $I$ is a module of $T$. Moreover, since $T(u,x) \neq T(u,y)$, we obtain $T(v,x) \neq T(v,y)$ and hence $T'(v,x) \neq T'(v,y)$, which contradicts that $M$ is a module of $T'$.
        \end{proof}
     We are now ready to prove Theorem~\ref{thm principal2}.
    \begin{proof}[Proof of Theorem~\ref{thm principal2}]
		Let $T$ be a tournament of order at least $5$ that is not transitive of even order. The theorem obviously holds when $T$ is indecomposable. Hence, suppose that $T$ is decomposable, i.e. $\Delta(T) \geq 2$. By Theorem~\ref{not tr}, ${\rm tr}(T) \neq \varnothing$.  If $\Delta(T) \geq 3$, then for every $X \in {\rm tr}(T)$, ${\rm Inv}(T,X)$ is indecomposable by Theorem~\ref{pinv}. Hence suppose $\Delta(T) =2$. By Theorem~\ref{deltan}, there exists $X \in \binom{V(T)}{2}$ such that ${\rm Inv}(T,X)$ is indecomposable. By Assertions~3 and 4 of Fact~\ref{rinv}, $X$ is a transversal of ${\rm mc}(T)$ and $\mathcal{M}(T)$ is strictly bipartite by $X$. Moreover, since $|X| = \Delta(T)= \nu({\rm mc}(T))$, the transversal $X$ of ${\rm mc}(T)$ is minimum and exact. Equivalently $X \in {\rm tr}(T)$. Hence, there always exists $X \in {\rm tr}(T)$ such that ${\rm Inv}(T, X)$ is indecomposable. Moreover, since $X$ is a minimum transversal of ${\rm mc}(T)$ and $|X| = \Delta(T)$, it follows from Assertion~3 of Fact~\ref{rinv} that $\delta'(T) = \Delta(T)$. This completes the proof.  
 \end{proof} 
 \section{Some additional remarks}
 Suppose that the tournament $T$ in Theorems~\ref{thm principal} and \ref{thhypergraphe} is transitive of even order with at least four vertices. In this instance, it is easy to see that $\mathcal{M}(T)$ is no longer strictly bipartite. It follows that there does not exist a subset $X$ of $V(T)$ such that ${\rm Inv}(T,X)$ is indecomposable (see Assertion~4 of Fact~\ref{rinv}). Thus, Theorems~\ref{thm principal} and \ref{thhypergraphe} fail to hold for these tournaments. However, it is easy to verify that ${\rm mc}(T)$ still has the K\"{o}nig property.
  \begin{cor} 
 	For any tournament $T$, ${\rm mc}(T)$ has the K\"{o}nig property.
 \end{cor}	
 On the other hand, Theorem~\ref{pinv} does not hold for the class $\mathcal{T}$ of the tournaments $T$ such $\Delta(T) = 2$. To see this, let us consider the tournament $T_n$ defined on $V(T_n) = \{0, \ldots, n-1\}$, where $n \geq 6$, in the following manner : 
 \begin{itemize}
 	\item $T_n-(n-1) = {\rm Inv}(\underline{n-1}, \{(i,i+1) : 0 \leq i \leq n-3\})$,
 	\item $T_n(n-1, \{0, \ldots, n-2\}) = 1$.
 \end{itemize}
It is easy to verify that $T_n-(n-1)$ is indecomposable, and that $\{0, \ldots, n-2\}$ is the unique nontrivial module of $T_n$. Therefore ${\rm mc}(T_n) = \{\{0, \ldots, n-2\}, \{n-1\}\}$, and hence $\Delta(T_n) =2$. It follows that $\{1, n-1\} \in {\rm tr}(T_n)$. But the tournament ${\rm Inv}(T_n, \{1,n-1\})$ is decomposable because $\{0,n-1\}$ is a nontrivial module of it. Therefore,
$T_n \in \mathcal{T} \cap \mathcal{R}$, where $\mathcal{R}$ is the class of the tournaments $T$ for which there exists $R \in {\rm tr}(T)$ such that ${\rm Inv}(T,R)$ is decomposable. Thus $\varnothing \neq \mathcal{R} \subseteq \mathcal{T}$. However, there exist tournaments $T$ of arbitrary large order, such that $\Delta(T)=2$ and for every $R \in {\rm tr}(T)$, ${\rm Inv}(T,R)$ is indecomposable (see e.g., \cite[Discussion]{Index2}). Thus $\mathcal{R} \varsubsetneq \mathcal{T}$. This discussion leads to the problem of characterization of the tournaments of the class $\mathcal{R}$.	   

We now explain how Theorem~\ref{thm principal} may be seen as a stronger version of Theorem~\ref{thhypergraphe}.

\begin{lem} \label{lemrem}
	Let $T$ be a tournament and $R$ a subset of $V(T)$. If $|R| = \Delta(T)$ and ${\rm Inv}(T,R)$ is indecomposable, then $R \in {\rm tr}(T)$.
	\end{lem}
\begin{proof}
	Suppose that $|R| = \Delta(T)$ and ${\rm Inv}(T,R)$ is indecomposable.
	By Assertions~3 and 4 of Fact~\ref{rinv}, $R$ is a transversal of ${\rm mc}(T)$ and $\mathcal{M}(T)$ is strictly bipartite by $R$. Since $|R| = \Delta(T)$, the transversal $R$ of ${\rm mc}(T)$ is minimum. Suppose to the contrary that the minimum transversal $R$ of ${\rm mc}(T)$ is not exact. In this instance, there exists $I \in {\rm mc}(T)$ such that $|I \cap R| \geq 2$. By minimality of $R$, we have $o_T(I) \neq 0$. It follows from Lemma~\ref{degre G_T} that $I$ is a twin of $T$, and in particular $I = I \cap R$. It follows that $I$ is also a twin of ${\rm Inv}(T,R)$, which contradicts that ${\rm Inv}(T,R)$ is indecomposable. We conclude that $R \in {\rm tr}(T)$.
\end{proof}
By Lemma~\ref{lemrem}, Theorem~\ref{thm principal} implies Theorem~\ref{thhypergraphe} for tournaments with at least five vertices. Moreover, by Lemma~\ref{lemrem} and Theorem~\ref{pinv}, Theorems~\ref{thm principal} and \ref{thhypergraphe} are equivalent for the tournaments $T$ such that $\Delta(T) \neq 2$. 

{}

\end{document}